\documentclass[12pt]{article}
\usepackage{amssymb}
\usepackage{amsmath}
\usepackage{amsthm}
\usepackage{color}
\usepackage{enumitem}
\usepackage{geometry}
\usepackage{graphicx}
\usepackage[sort&compress]{natbib}     % <-- new

\makeatletter
	\@addtoreset{equation}{section}
\makeatother

\setlist[enumerate]{itemsep=0pt}

\let\originalleft\left
\let\originalright\right
\renewcommand{\left}{\mathopen{}\mathclose\bgroup\originalleft}
\renewcommand{\right}{\aftergroup\egroup\originalright}

\begin{document}

\newcommand{\bO}{{\bf 0}}
\newcommand\cF{\mathcal{F}}
\newcommand\cL{\mathcal{L}}
\newcommand\cM{\mathcal{M}}
\newcommand\cN{\mathcal{N}}
\newcommand{\rD}{{\rm D}}
\newcommand{\re}{{\rm e}}
\newcommand{\ri}{{\rm i}}
\newcommand{\ee}{\varepsilon}

\newtheorem{theorem}{Theorem}[section]
\newtheorem{corollary}[theorem]{Corollary}
\newtheorem{lemma}[theorem]{Lemma}
\newtheorem{proposition}[theorem]{Proposition}

\theoremstyle{definition}
\newtheorem{definition}{Definition}[section]
\newtheorem{example}[definition]{Example}

\theoremstyle{remark}
\newtheorem{remark}{Remark}[section]

% keywords: piecewise-smooth; bifurcation; dynamics; piecewise-linear

% MSC codes:
%		34A36 -- Discontinuous Equations
%		34D20 -- Stability of solutions to ordinary differential equations

\title{
The stability of boundary equilibria of three-dimensional Filippov systems.
}
\author{
D.J.W.~Simpson\\\\
School of Mathematical and Computational Sciences\\
Massey University\\
Palmerston North, 4410\\
New Zealand
}
\maketitle

\begin{abstract}

For three-dimensional piecewise-smooth systems of ordinary differential equations, this paper characterises the stability of points that belong to a switching surface and are equilibria of exactly one of the two neighbouring pieces of the system. Stability is challenging to characterise when nearby orbits repeatedly switch between regular motion on one side of the switching surface, and sliding motion on the switching surface, as defined via Filippov's convention. We prove that in this case stability is governed by the behaviour of a global reinjection mechanism of a four-parameter family of piecewise-linear hybrid systems, and perform a detailed numerical study of this family.

\end{abstract}

%===============================================================================
\section{Introduction}
\label{sec:intro}

Dynamical systems with multiple modes of evolution are naturally modelled by piecewise-smooth differential equations.
Each piece of the system applies in some region of phase space,
with a boundary comprised of one or more codimension-one {\em switching surfaces}.
In some models, the system state is permitted to {\em slide} along a switching surface.
For example, in models of relay control and threshold control systems,
sliding motion represents the idealised limit that the reaction time of the controller is zero \cite{JoRa99,TaLi12}.
In mechanical models for which the friction between in-contact objects is modelled by Coulomb's law,
sliding motion corresponds to the objects being stuck together \cite{CaGi06,DiKo03,Sh86}.

Mathematically, sliding motion can be defined by averaging the neighbouring smooth pieces of the differential equations.
In this paper, we use the standard Filippov framework \cite{Fi88} for defining solutions in a way accommodates sliding motion.

If an equilibrium of a Filippov system does not belong to a switching surface,
then, except in special cases, % e.g.~at bifurcations,
the stability of the equilibrium is governed by the eigenvalues of the Jacobian matrix evaluated at the equilibrium \cite{Me07}.
For equilibria on switching surfaces, stability is more complicated.
This problem has been heavily studied for equilibria that are zeros of each neighbouring piece of the differential equations
due to their ubiquity in control applications \cite{Jo03,Li03,LiAn09}.
The standard approach is to employ algorithm to search for a Lyapunov function whose existence implies stability.
However, these algorithms rarely characterise stability, as their failure to produce a Lyapunov function usually
does not imply that the equilibrium is unstable.
Extra considerations are needed to accommodate sliding motion \cite{DeRo14,IeTr20,Su10},
and alternate techniques have been developed to handle equilibria
that lie at the intersection of multiple switching surfaces \cite{AkAr09,IwHa06,XuHu10}.
See also \cite{VaLe04} for the stability of a continuum of equilibria on a switching surface.

This paper treats points on isolated switching surfaces that are equilibria of exactly one neighbouring piece of the system.
Such points occur at boundary equilibrium bifurcations \cite{DiBu08}
where an equilibrium of one piece of the system collides with a switching surface as parameters are varied.
Such bifurcations can act as tipping points by destroying a local attractor \cite{Si25e}.
However, this cannot occur if the equilibrium at the bifurcation is asymptotically stable \cite{DiNo08},
and this motivates the present work.

As an example, Fig.~\ref{fig:A} shows a phase portrait of a three-dimensional system
\begin{equation}
\dot{x} = \begin{cases}
f^L(x), & H(x) < 0, \\
f^R(x), & H(x) > 0,
\end{cases}
\label{eq:f}
\end{equation}
with a boundary equilibrium $x^*$ (throughout this paper dots denote differentiation with respect to time, $t$).
Locally, the switching surface $\Sigma$ (where $H(x) = 0$) is divided into an {\em attracting sliding region},
where the vector fields $f^L$ and $f^R$ are both directed toward $\Sigma$, and a {\em crossing region},
where $f^L$ and $f^R$ have the same direction relative to $\Sigma$.
The equilibrium $x^*$ lies on the boundary $\Gamma \subset \Sigma$ of these regions.
Nearby orbits switch from regular motion to sliding motion whenever they reach $\Sigma$,
and from sliding motion to regular motion whenever they reach $\Gamma$.

%%%%%%%%%%%%%%%%%%%%%%%%%%%%%%%%%%%%%%%%%%%%%%%%%%%%%%%%%%%%%%%%%%%%%%%%%%%%%%%%%%%%%%%%%%%%%%%%%%%%
\begin{figure}[b!]
\begin{center}
\includegraphics[width=8.8cm]{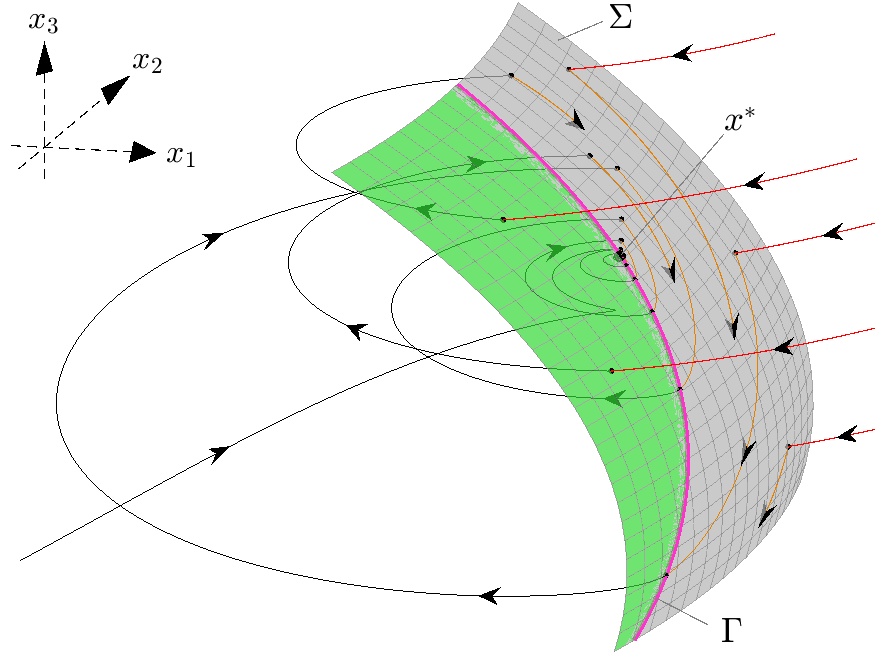}
\caption{
A phase portrait of a three-dimensional Filippov system of the form \eqref{eq:f}.
The switching surface $\Sigma$ (where $H(x) = 0$) consists of an attracting sliding region (grey),
a crossing region (green), and the boundary $\Gamma$ between these regions (pink).
Orbits are coloured black where they evolve under $f^L$,
red where they evolve under $f^R$,
and orange where they evolve under the sliding vector field $f^S$.
The boundary equilibrium $x^* \in \Gamma$ obeys $f^L(x^*) = \bO$ and $f^R(x^*) \ne \bO$.
Below $x^*$, points on $\Gamma$ are visible folds (see \S\ref{sub:pointsOfTangency});
above $x^*$, points on $\Gamma$ are invisible folds.
\label{fig:A}
} 
\end{center}
\end{figure}
%%%%%%%%%%%%%%%%%%%%%%%%%%%%%%%%%%%%%%%%%%%%%%%%%%%%%%%%%%%%%%%%%%%%%%%%%%%%%%%%%%%%%%%%%%%%%%%%%%%%

In two dimensions the stability of boundary equilibria in generic situations
is easy to ascertain because sliding motion is one-dimensional
and nearby orbits cannot repeatedly switch between regular motion and sliding motion \cite{HoHo16}.
To handle three dimensions we leverage earlier work \cite{Si18d,Si21}
to justify replacing $f^L$ and the sliding vector field with their linearisations.
We also convert $\rD f^L$ to a companion matrix via a change of coordinates.
These considerations lead us to the piecewise-linear form
\begin{equation}
\dot{y} = \begin{cases}
g^L(y), & \text{until $y_1 = 0$}, \\
g^S(y), & \text{until $y_2 = 0$},
\end{cases}
\label{eq:pwl}
\end{equation}
where $y = (y_1,y_2,y_3)$, and
\begin{align}
g^L(y) &= \begin{bmatrix} a-1 & 1 & 0 \\ a-b & 0 & 1 \\ -b & 0 & 0 \end{bmatrix} y, &
g^S(y) &= \begin{bmatrix} 0 & 0 & 0 \\ 0 & c & 1 \\ 0 & -d & 0 \end{bmatrix} y,
\label{eq:gLgS}
\end{align}
where $a, b, c, d \in \mathbb{R}$ are parameters.
Notice \eqref{eq:pwl} is a hybrid system, not a Filippov system.
Orbits switch from following $g^L$ to following $g^S$ when they reach the coordinate plane $y_1 = 0$, denoted $\tilde{\Sigma}$,
and switch from following $g^S$ to following $g^L$ when they reach the line $y_1 = y_2 = 0$, denoted $\tilde{\Gamma}$.
If orbits repeatedly visit $\tilde{\Gamma}$, as in Fig.~\ref{fig:B},
then the map $\zeta$ of first return to $\tilde{\Gamma}$ is well-defined.
Moreover, $\zeta$ is linear and forward orbits can only reach points $(0,0,z) \in \tilde{\Gamma}$ for which $z < 0$.
Thus $\zeta(z) = \Lambda z$, where
\begin{equation}
\Lambda = -\zeta(-1),
\label{eq:Lambda}
\end{equation}
and this value governs the stability of the boundary equilibrium.
A similar use of maps to characterise stability is described by Eldem and \"{O}ner \cite{ElOn15}
when $x^*$ is an equilibrium of both pieces of the system,
and by Gon\c{c}alves {\em et al.}~\cite{GoMe03} for a general class of control systems.

%%%%%%%%%%%%%%%%%%%%%%%%%%%%%%%%%%%%%%%%%%%%%%%%%%%%%%%%%%%%%%%%%%%%%%%%%%%%%%%%%%%%%%%%%%%%%%%%%%%%
\begin{figure}[b!]
\begin{center}
\includegraphics[width=8.8cm]{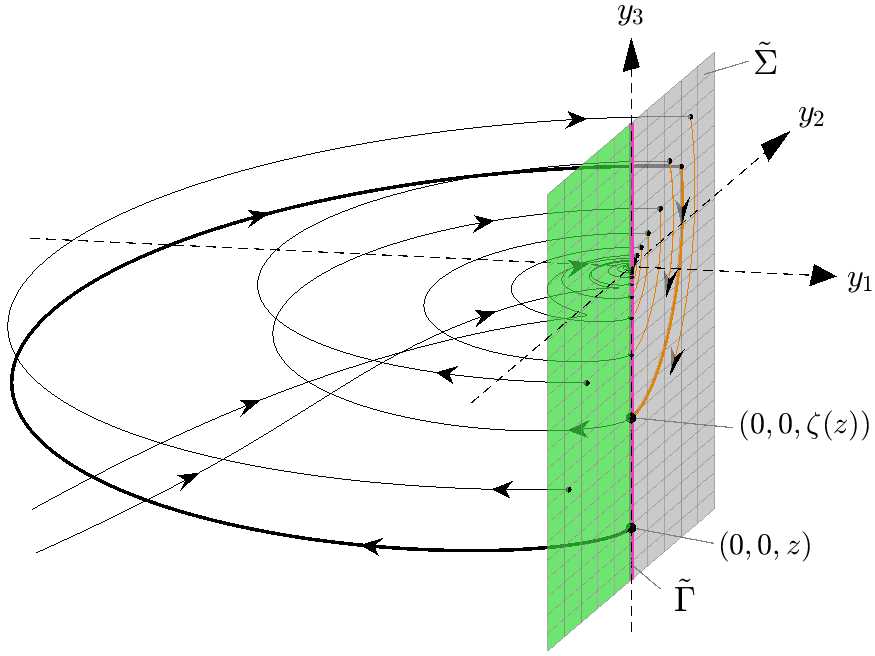}
\caption{
A phase portrait of the piecewise-linear system \eqref{eq:pwl}--\eqref{eq:gLgS} with $(a,b,c,d) = (-0.2,5,-0.2,3)$.
Orbits evolve under $g^L$ and are coloured black until reaching $\tilde{\Sigma}$ (where $y_1 = 0$).
Here they switch to evolution under $g^S$ and are coloured orange until reaching $\tilde{\Gamma}$ (where $y_1 = y_2 = 0$).
Given $z < 0$, we write $(0,0,\zeta(z))$ for the next point at which the forward orbit of $(0,0,z)$ intersects $\tilde{\Gamma}$.
\label{fig:B}
} 
\end{center}
\end{figure}
%%%%%%%%%%%%%%%%%%%%%%%%%%%%%%%%%%%%%%%%%%%%%%%%%%%%%%%%%%%%%%%%%%%%%%%%%%%%%%%%%%%%%%%%%%%%%%%%%%%%

The remainder of this paper is organised as follows.
In \S\ref{sec:main} we state the main result (Theorem \ref{th:main}) for a boundary equilibrium $x^*$
of a general three-dimensional Filippov system of the form \eqref{eq:f}.
This result states that in cases for which orbits near $x^*$ repeatedly visit $\Gamma$,
the stability of $x^*$ is governed by the sign of $\ln(\Lambda)$ for the corresponding hybrid system \eqref{eq:pwl}--\eqref{eq:gLgS},
while in other cases the stability is governed by the eigenvalues of $\rD f^L(x^*)$ and $\rD f^S(x^*)$, where $f^S$ is the sliding vector field.

In \S\ref{sec:prelim} we clarify some terminology for Filippov systems,
then in \S\ref{sec:red} perform the reduction from \eqref{eq:f} to \eqref{eq:pwl}--\eqref{eq:gLgS}.
Section \ref{sec:proof} contains a proof of Theorem \ref{th:main},
then \S\ref{sec:numerics} presents a numerical exploration of how the value of $\Lambda$
varies with the values of $a$, $b$, $c$, and $d$.
Notably, $x^*$ can be unstable when all eigenvalues of $\rD f^L(x^*)$ and $\rD f^S(x^*)$ have negative real-part,
except the trivial zero eigenvalue of $\rD f^S(x^*)$.
Further, $x^*$ can be asymptotically stable when $\rD f^L(x^*)$ and $\rD f^S(x^*)$ both have eigenvalues with positive real-part.
Concluding remarks are provided in \S\ref{sec:conc}.

%===============================================================================
\section{Main result}
\label{sec:main}

Consider a Filippov system \eqref{eq:f}, where $f^L, f^R : \mathbb{R}^3 \to \mathbb{R}^3$ are $C^1$
and $H : \mathbb{R}^3 \to \mathbb{R}$ is $C^2$.
The switching surface is the set
\begin{equation}
\Sigma = \left\{ x \in \mathbb{R}^3 \,\middle|\, H(x) = 0 \right\}.
\label{eq:Sigma}
\end{equation}
Suppose $x^* \in \Sigma$ is such that $f^L(x^*) = \bO$, and define
\begin{align}
p &= \nabla H(x^*), \label{eq:p} \\
q &= f^R(x^*), \label{eq:q} \\
A &= \rD f^L(x^*). \label{eq:A} \\
B &= \left( I - \frac{q p^{\sf T}}{p^{\sf T} q} \right) A, \label{eq:B}
\end{align}
which in \eqref{eq:B} assumes $p^{\sf T} q \ne 0$
(meaning that $f^R$ is not tangent to $\Sigma$ at $x^*$).
The vector $p$ is a normal vector for $\Sigma$ at $x^*$,
the matrix $A$ is the Jacobian matrix of $f^L$ at $x^*$,
and the matrix $B$ is the Jacobian matrix of the sliding vector field $f^S$ (defined in \S\ref{sub:fS}) at $x^*$.
Notice $p^{\sf T} B = \bO^{\sf T}$, so $0$ is an eigenvalue of $B$.
Also define the {\em observability matrix}
\begin{equation}
\Phi = \begin{bmatrix} p^{\sf T} \\ p^{\sf T} A \\ p^{\sf T} A^2 \end{bmatrix}.
\label{eq:Phi}
\end{equation}

%...............................................................................
\begin{theorem}
Suppose $p^{\sf T} q < 0$ and $\det(\Phi) \ne 0$.
\begin{enumerate}
\item
If $A$ or $B$ has a positive eigenvalue, then $x^*$ is unstable.
\item
If $A$ has three distinct negative eigenvalues,
and the non-zero eigenvalues of $B$ are either complex or both negative,
then $x^*$ is asymptotically stable.
\item
Suppose $A$ has eigenvalues $\alpha \pm \ri \beta$ and $-\gamma$,
where $\alpha \in \mathbb{R}$, $\beta > 0$, and $\gamma > 0$,
and the non-zero eigenvalues $\lambda^S_1$ and $\lambda^S_2$ of $B$ are either complex or both negative.
Let
\begin{equation}
\begin{aligned}
a &= \frac{2 \alpha}{\gamma}, & \qquad \qquad
c &= \frac{\lambda^S_1 + \lambda^S_2}{\gamma}, \\
b &= \frac{\alpha^2 + \beta^2}{\gamma^2}, & \qquad \qquad
d &= \frac{\lambda^S_1 \lambda^S_2}{\gamma^2},
\end{aligned}
\label{eq:pwlParams}
\end{equation}
and let $\Lambda$ be given by \eqref{eq:Lambda}
for the piecewise-linear system \eqref{eq:pwl}--\eqref{eq:gLgS}.
If $\Lambda$ is undefined or $\Lambda < 1$ then $x^*$ is asymptotically stable,
while if $\Lambda > 1$ then $x^*$ is unstable.
\end{enumerate}
\label{th:main}
\end{theorem}

%...............................................................................
\begin{remark}
The condition $\det(\Phi) \ne 0$ is necessary for the genericity of the tangency curve $\Gamma$ (see \S\ref{sec:prelim}).
This condition also ensures that the companion matrix $\rD g^L$ in \eqref{eq:gLgS} can realised via a change of coordinates (see Proposition \ref{pr:normalForm}).
By the Popov-Belevitch-Hautus observability test \cite{So98},
$\det(\Phi) \ne 0$ is equivalent to the statement that $A$ has no eigenvector orthogonal to $p$.
\label{re:pbh}
\end{remark}

%...............................................................................
\begin{remark}
The condition $p^{\sf T} q < 0$ ensures that $f^R$ is directed toward $\Sigma$ in a neighbourhood of $x^*$.
If $p^{\sf T} q > 0$ then $f^R$ is directed away from $\Sigma$, and $x^*$ is unstable.
\label{re:pTq}
\end{remark}

%===============================================================================
\section{Preliminaries}
\label{sec:prelim}

In this section we review fundamental aspects of Filippov systems relevant to
systems satisfying the conditions of Theorem \ref{th:main}.
For more details refer to Filippov \cite{Fi88} or the books \cite{DiBu08,Je18b}.

%-------------------------------------------------------------------------------
\subsection{Sliding and crossing regions}

Consider a three-dimensional Filippov system \eqref{eq:f}
for which $f^L$ and $f^R$ are $C^1$ and $H$ is $C^2$.
If $p \ne \bO$, where $p$ is the gradient vector of $H$ at $x^*$ \eqref{eq:p},
then, in a neighbourhood of $x^*$, the switching surface $\Sigma$ \eqref{eq:Sigma}
is $C^2$ by the regular value theorem \cite{Hi76}.
Let
\begin{align}
\Omega_L &= \left\{ x \in \mathbb{R}^3 \,\middle|\, H(x) < 0 \right\}, \\
\Omega_R &= \left\{ x \in \mathbb{R}^3 \,\middle|\, H(x) > 0 \right\},
\end{align}
be the left and right subdomains.
Also let 
\begin{align}
v_L(x) &= \nabla H(x)^{\sf T} f^L(x), &
v_R(x) &= \nabla H(x)^{\sf T} f^R(x),
\end{align}
denote the rate of change of $H$ along orbits following $f^L$ and $f^R$.

%...............................................................................
\begin{definition}
A subset $U \subset \Sigma$ is
\begin{enumerate}
\item
a {\em crossing region} if $v_L(x) v_R(x) > 0$ for all $x \in U$,
\item
an {\em attracting sliding region} if $v_L(x) > 0$ and $v_R(x) < 0$ for all $x \in U$, and
\item
a {\em repelling sliding region} if $v_L(x) < 0$ and $v_R(x) > 0$ for all $x \in U$.
\end{enumerate}
\label{df:regions}
\end{definition}

%-------------------------------------------------------------------------------
\subsection{Points of tangency}
\label{sub:pointsOfTangency}

In Theorem \ref{th:main}, $f^L(x^*) = \bO$ and $p^{\sf T} q < 0$, so $v_L(x^*) = 0$ and $v_R(x^*) < 0$.
Moreover, $\Phi$ is invertible, thus the vectors $p^{\sf T}$ and $\nabla v_L(x^*)^{\sf T} = p^{\sf T} A$ are linearly independent.
This implies that
\begin{equation}
\Gamma = \left\{ x \in \Sigma \,\middle|\, v_L(x) = 0 \right\},
\label{eq:}
\end{equation}
is a $C^1$ curve in a neighbourhood of $x^*$ by the regular value theorem \cite{Hi76}.
This curve consists of points where $f^L$ is tangent to $\Sigma$,
and splits $\Sigma$ into an attracting sliding region and a crossing region, see Fig.~\ref{fig:A}.
To classify points on $\Gamma$, let
\begin{equation}
w_L(x) = \nabla v_L(x)^{\sf T} f^L(x),
\label{eq:aL}
\end{equation}
denote the second derivative of $H$ with respect to $t$ for orbits following $f^L$.
Notice $w_L(x^*) = 0$ and $\nabla w_L(x^*)^{\sf T} = p^{\sf T} A^2$.
Since $\Phi$ is invertible,
on $\Gamma$ we have $w_L(x) < 0$ on one side of $x^*$
and $w_L(x) > 0$ on the other side of $x^*$.
At points on $\Gamma$ with $w_L(x) < 0$,
the forward orbit of $x$ under $f^L$ immediately enters $\Omega_L$ and is a Filippov solution of \eqref{eq:f}.
Such points are termed {\em visible folds}.
If instead $w_L(x) > 0$, $x$ is termed an {\em invisible fold}.

%-------------------------------------------------------------------------------
\subsection{Filippov solutions, semi-flows, and the stability of equilibria}

A {\em Filippov solution} to \eqref{eq:f} is an absolutely continuous function $\phi(t)$
for which $\dot{\phi}(t) \in \cF(\phi(t))$ for almost all $t$, where
\begin{equation}
\cF(x) = \begin{cases}
\{ f^L(x) \}, & H(x) < 0, \\
\left\{ (1-\theta) f^L(x) + \theta f^R(x) \,\middle|\, 0 \le \theta \le 1 \right\}, & H(x) = 0, \\
\{ f^R(x) \}, & H(x) > 0.
\end{cases}
\label{eq:cF}
\end{equation}
By Theorem 2 of Filippov \cite[\S 10]{Fi88},
Filippov solutions of \eqref{eq:f} exist and are unique forwards in time
throughout regions of phase space for which either $v_L(x) > 0$ or $v_R(x) < 0$ at any point in $\Sigma$.
For \eqref{eq:f} satisfying the conditions of Theorem \ref{th:main},
we have $v_R(x^*) = p^{\sf T} q < 0$, which implies $v_R(x) < 0$ in a neighbourhood $\cN$ of $x^*$.
Thus \eqref{eq:f} has forward uniqueness throughout $\cN$.

When forward uniqueness holds, the collection of all solutions is a semi-flow.
A {\em semi-flow} is a continuous function $\varphi_t(x)$
that obeys $\varphi_0(x) = x$ (initial condition),
and $\varphi_{t_2} \left( \varphi_{t_1}(x) \right) = \varphi_{t_1 + t_2}(x)$ (group property) for all $x$ and all $t_1, t_2 \ge 0$
for which both sides of the group property are defined.

We now define stability of equilibria in terms of semi-flows.

%...............................................................................
\begin{definition}
An {\em equilibrium} of a semi-flow $\varphi_t(x)$ is
a point $x^*$ for which $\varphi_t(x^*) = x^*$ for all $t \ge 0$, and is
\begin{enumerate}
\item 
{\em Lyapunov stable} if for every neighbourhood $\cN$ of $x^*$
there exists a subneighbourhood $\cM$ with the property
that $\varphi_t(x) \in \cN$ for all $x \in \cM$ and $t \ge 0$,
\item
{\em asymptotically stable} if it is Lyapunov stable and
$\varphi_t(x) \to x^*$ as $t \to \infty$ for all $x$ in a neighbourhood of $x^*$, and
\item
{\em unstable} if it is not Lyapunov stable.
\end{enumerate}
\label{df:stable}
\end{definition}

%-------------------------------------------------------------------------------
\subsection{The sliding vector field}
\label{sub:fS}

The vector $(1-\theta) f^L + \theta f^R$ appearing in \eqref{eq:cF} is a convex combination of $f^L$ and $f^R$.
If $v_L \ne v_R$, then this vector is tangent to $\Sigma$ if and only if $\theta = \frac{v_L}{v_L - v_R}$, in which case
the vector is
\begin{equation}
f^S(x) = \frac{v_L(x) f^R(x) - v_R(x) f^L(x)}{v_L(x) - v_R(x)}.
\label{eq:fS}
\end{equation}
If a Filippov solution is constrained to $\Sigma$,
then it follows \eqref{eq:fS}.
Consequently \eqref{eq:fS} is known as the {\em sliding vector field}.
It is a simple exercise in calculus to show from \eqref{eq:fS} that
\begin{equation}
\rD f^S(x^*) = B,
\label{eq:B2}
\end{equation}
where $B$ is given by \eqref{eq:B}.

At $x \in \Gamma$, we have $v_L(x) = 0$, and so $f^S(x) = f^L(x)$.
Thus $w_L(x)$ is the rate of change of $v_L$ along the orbit following $f^S$.
Hence at visible folds, $f^S$ is directed out of the attracting sliding region,
while at invisible folds, $f^S$ is directed into the attracting sliding region.

%...............................................................................
\begin{remark}
To summarise, for a system \eqref{eq:f} satisfying the conditions of Theorem \ref{th:main},
forward evolution in the sense of Filippov is unique in a neighbourhood of $x^*$.
Orbits follow $f^L$ while in $\Omega_L$,
$f^R$ while in $\Omega_R$,
and slide following $f^S$ while belonging to the attracting sliding subset of $\Sigma$.
Orbits switch from sliding motion to motion in $\Omega_L$ at visible folds.
These are are points in $\Gamma$ on the side of $x^*$ for which $w_L(x) < 0$.
\end{remark}

%===============================================================================
\section{Reductions}
\label{sec:red}

In this section we perform a series of reductions to convert
a general system \eqref{eq:f} satisfying case (iii) of Theorem \ref{th:main}
to the piecewise-linear form \eqref{eq:pwl}--\eqref{eq:gLgS}.
The reductions are achieved in such a way that asymptotic stability of $y = \bO$ for \eqref{eq:pwl}--\eqref{eq:gLgS}
implies asymptotic stability of $x = x^*$ for \eqref{eq:f}.

%-------------------------------------------------------------------------------
\subsection{The leading-order approximation}

We first Taylor expand $H$, $f^L$, and $f^R$ about $x = x^*$.
Since $H(x^*) = 0$ and $\nabla H(x^*) = p$, the leading-order term of $H$ is $p^{\sf T} (x-x^*)$.
Since $f^L(x^*) = \bO$ and $\rD f^L(x^*) = A$, the leading-order term of $f^L$ is $A (x-x^*)$.
Finally $f^R(x^*) = q \ne \bO$, so the leading-order term of $f^R$ is the constant vector $q$.
By replacing $H$, $f^L$, and $f^R$ with their leading-order terms,
and shifting $x^*$ to the origin via the translational change of coordinates $u = x - x^*$,
the system \eqref{eq:f} becomes
\begin{equation}
\dot{u} = \begin{cases}
A u, & p^{\sf T} u < 0, \\
q, & p^{\sf T} u > 0.
\end{cases}
\label{eq:fApprox}
\end{equation}
The following result of \cite{Si21} %(see Theorems 2.1 and 2.2)
shows that the nonlinear terms omitted to produce \eqref{eq:fApprox}
cannot break the asymptotic stability of $x^*$.

%...............................................................................
\begin{proposition}
If $p^{\sf T} q < 0$ and $u = \bO$ is an asymptotically stable equilibrium of \eqref{eq:fApprox},
then $x = x^*$ is an asymptotically stable equilibrium of \eqref{eq:f}.
\label{pr:hots}
\end{proposition}

%-------------------------------------------------------------------------------
\subsection{The three-dimensional boundary equilibrium normal form}

Since $0$ is an eigenvalue of $B$,
the characteristic polynomials of $A$ and $B$ can be written as
\begin{align}
\det(\lambda I - A) &= \lambda^3 - \tau_L \lambda^2 + \sigma_L \lambda - \delta_L \,, \\
\det(\lambda I - B) &= \lambda^3 - \tau_S \lambda^2 - \delta_S \lambda,
\end{align}
where $\tau_L, \sigma_L, \delta_L, \tau_S, \delta_S \in \mathbb{R}$.
The following result is a consequence of Theorem 7 of \cite{Si18d}
and gives conditions under which \eqref{eq:fApprox} can be converted to
\begin{equation}
\dot{y} = \begin{cases}
\begin{bmatrix} \tau_L y_1 + y_2 \\ -\sigma_L y_1 + y_3 \\ \delta_L y_1 \end{bmatrix}, & y_1 < 0, \\
\begin{bmatrix} -1 \\ \tau_S \\ -\delta_S \end{bmatrix}, & y_1 > 0.
\end{cases}
\label{eq:normalForm}
\end{equation}
The system \eqref{eq:normalForm} is the boundary equilibrium normal form
in three dimensions evaluated at a boundary equilibrium bifurcation.

%...............................................................................
\begin{proposition}
If $p^{\sf T} q < 0$ and $\det(\Phi) \ne 0$,
then there exists a linear change of coordinates that converts \eqref{eq:fApprox} to \eqref{eq:normalForm}.
\label{pr:normalForm}
\end{proposition}

%-------------------------------------------------------------------------------
\subsection{A time scaling to linearise the sliding vector field}

The attracting sliding region of \eqref{eq:normalForm} is
$\left\{ y \in \mathbb{R}^3 \,\middle|\, y_1 = 0,\, y_2 > 0 \right\}$.
On this region the sliding vector field is
\begin{equation}
\dot{y} = \frac{1}{y_2 + 1} \begin{bmatrix} 0 \\ \tau_S y_2 + y_3 \\ -\delta_S y_2 \end{bmatrix}.
\label{eq:normalFormSVF}
\end{equation}
The factor $\frac{1}{y_2 + 1}$ is positive throughout the attracting sliding region
and can be removed by rescaling time.
Orbits of the linear system
\begin{equation}
\dot{y} = \begin{bmatrix} 0 \\ \tau_S y_2 + y_3 \\ -\delta_S y_2 \end{bmatrix},
\label{eq:normalFormSVFScaled}
\end{equation}
follow the same paths as those of \eqref{eq:normalFormSVF}, but with a different temporal parameterisation.
Formally this is justified by Theorem 6 of Filippov \cite[\S 9]{Fi88}.

Patching \eqref{eq:normalFormSVFScaled} to the left piece of \eqref{eq:normalForm}
produces the piecewise-linear hybrid system
\begin{equation}
\dot{y} = \begin{cases}
\begin{bmatrix} \tau_L y_1 + y_2 \\ -\sigma_L y_1 + y_3 \\ \delta_L y_1 \end{bmatrix}, & \text{until $y_1 = 0$}, \\
\begin{bmatrix} 0 \\ \tau_S y_2 + y_3 \\ -\delta_S y_2 \end{bmatrix}, & \text{until $y_2 = 0$}.
\end{cases}
\label{eq:pwl2}
\end{equation}
The switching conditions in \eqref{eq:pwl2} are such that orbits of \eqref{eq:pwl2}
match one-to-one to orbits of \eqref{eq:normalForm} with $y_1 \le 0$.
Moreover, \eqref{eq:pwl2} induces a unique semi-flow on $\left\{ y \in \mathbb{R}^3 \,\middle|\, y_1 \le 0 \right\}$
for all $t \ge 0$ and for any values of the five parameters.

%-------------------------------------------------------------------------------
\subsection{A time scaling to eliminate one parameter}

We now consider case (iii) of Theorem \ref{th:main}
and perform a second time scaling that leads to a reduction from the five-parameter form \eqref{eq:pwl2}
to the four-parameter form \eqref{eq:pwl}--\eqref{eq:gLgS}.

In case (iii), $A$ has eigenvalues $\alpha \pm \ri \beta$ and $-\gamma$,
where $\alpha \in \mathbb{R}$, $\beta > 0$, and $\gamma > 0$,
and $B$ has eigenvalues $0$ and $\lambda^S_1, \lambda^S_2 \in \mathbb{C}$.
By applying the linear time scaling $s = \gamma t$ to \eqref{eq:fApprox}, we obtain
\begin{equation}
\frac{d u}{d s} = \begin{cases}
\frac{A u}{\gamma}, & p^{\sf T} u < 0, \\
\frac{q}{\gamma}, & p^{\sf T} u > 0.
\end{cases}
\label{eq:fApproxScaled}
\end{equation}
The scaling does not alter the stability of $u = \bO$ because $\gamma > 0$.

The Jacobian matrix of the left piece of \eqref{eq:fApproxScaled} is
$\frac{A}{\gamma}$, while
the Jacobian matrix of the sliding vector field of \eqref{eq:fApproxScaled} is
$\left( I - \frac{q p^{\sf T}}{p^{\sf T} q} \right) \frac{A}{\gamma}$.
The coefficients of the characteristic polynomials of these matrices are
\begin{equation}
\begin{aligned}
\tau_L &= \frac{2 \alpha}{\gamma} - 1, & \qquad
\sigma_L &= -\frac{2 \alpha}{\gamma} + \frac{\alpha^2 + \beta^2}{\gamma^2}, & \qquad
\delta_L &= -\frac{\alpha^2 + \beta^2}{\gamma^2}, \\
\tau_S &= \frac{\lambda^S_1 + \lambda^S_2}{\gamma}, & \qquad
\delta_S &= \frac{\lambda^S_1 \lambda^S_2}{\gamma^2}.
\end{aligned}
\label{eq:normalFormParams}
\end{equation}
By substituting these formulas into \eqref{eq:pwl2}
we obtain \eqref{eq:pwl}--\eqref{eq:gLgS}, with $a$, $b$, $c$, and $d$ given by \eqref{eq:pwlParams}.

The time scaling $s = \gamma t$ was chosen so that
the Jacobian matrix of the left piece of \eqref{eq:fApproxScaled} has an eigenvalue of $-1$.
The transformation in Proposition \ref{pr:normalForm} affects a similarity transform
on this matrix, taking it to the Jacobian matrix $\rD g^L$ for \eqref{eq:gLgS}.
Similarity transforms do not alter eigenvalues,
thus $-1$ is an eigenvalue of $\rD g^L$ for all $a, b \in \mathbb{R}$.

%===============================================================================
\section{Main proof}
\label{sec:proof}

In this section we prove Theorem \ref{th:main}.
To prove $x^*$ is unstable in case (i),
we use an eigenvector associated with the unstable eigenvalue to show that \eqref{eq:f} has an orbit emanating from $x^*$.
To prove $x^*$ is asymptotically stable in case (ii),
we show that the forward orbit of $y = (0,0,-1)$ under $g^L$ never reintersects $y_1 = 0$ (see Appendix \ref{app:threeNegativeEigs})
and thus converges to $\bO$.
For case (iii) with $\Lambda < 1$ we use Proposition \ref{pr:hots} to obtain
the asymptotic stability of $x^*$ from that of $\bO$ for \eqref{eq:pwl}--\eqref{eq:gLgS}.	
But to prove instability in the case $\Lambda > 1$ we cannot use Proposition \ref{pr:hots},
so retain the higher order terms in our analysis.
This requires a spatial blow-up to prove that $\Lambda > 1$ implies the equilibrium is not Lyapunov stable.

To prove the asymptotic stability of $\bO$,
we usurp the linear homogeneity of \eqref{eq:pwl}--\eqref{eq:gLgS} and \eqref{eq:pwl2}.
These systems induce a unique semi-flow $\psi_t(y)$ on the left half-space
\begin{equation}
\Psi = \left\{ y \in \mathbb{R}^3 \,\middle|\, y_1 \le 0 \right\},
\label{eq:Psi}
\end{equation}
that is {\em linearly homogeneous} in the sense that
\begin{equation}
\psi_t(\nu y) = \nu \psi_t(y),
\label{eq:linearHomogeneity}
\end{equation}
for all $t \ge 0$, $y \in \Psi$, and $\nu \ge 0$.
This identity holds because each piece of \eqref{eq:pwl} or \eqref{eq:pwl2} is linear, as are the switching rules.
In Appendix \ref{app:asyStab}, we prove that if $\psi_t(y) \to \bO$ for all $y$ in a neighbourhood of $\bO$,
then $\bO$ is asymptotically stable.
With more effort one could show that $\bO$ is exponentially stable, as in Lasota and Strauss \cite[Theorem 1.2]{LaSt71},
but this is not needed for our purposes.

%...............................................................................
\begin{proof}[Proof of Theorem \ref{th:main}]
We prove (i), (ii), and (iii) in order.
\begin{enumerate}
\item
First suppose $A v = \lambda v$, where $\lambda > 0$ and $v \ne \bO$.
Then $\dot{x} = f^L(x)$ has an orbit $\xi(t)$ emanating from $x^*$ in the direction $v$ \cite{KaHa95}.
Notice $p^{\sf T} v \ne 0$, by Remark \ref{re:pbh}, so we can assume $p^{\sf T} v < 0$.
Thus $\xi(t) \in \Omega_L$, so $\xi(t)$ a Filippov solution of \eqref{eq:f}, hence $x^*$ is unstable.

Second suppose $B w = \lambda w$, where $\lambda > 0$ and $w \ne \bO$.
By \eqref{eq:B2}, $\dot{x} = f^S(x)$ has an orbit $\xi(t)$ emanating from $x^*$ in the direction $w$.
Since $p^{\sf T} B = \bO^{\sf T}$ by \eqref{eq:B}, we have $p^{\sf T} w = \frac{p^{\sf T} B w}{\lambda} = 0$.
Thus $A w \ne \lambda w$ by Remark \ref{re:pbh}.
But \eqref{eq:B} and $B w = \lambda w$ give
\begin{equation}
A w - \frac{q p^{\sf T} A w}{p^{\sf T} q} = \lambda w,
\nonumber
\end{equation}
thus $p^{\sf T} A w \ne 0$.
Thus $w$ is not tangent to $\Gamma$, so, substituting $w \mapsto -w$ if necessary,
$\xi(t)$ belongs to the attracting sliding region and is a Filippov solution of \eqref{eq:f}, hence $x^*$ is unstable.
\item
Here we work with the piecewise-linear system \eqref{eq:pwl2}.
This was obtained from \eqref{eq:f} by forming the truncated system \eqref{eq:fApprox},
then applying a linear change of coordinates (Proposition \ref{pr:normalForm}), and lastly rescaling time.
For \eqref{eq:pwl2}, forward orbits in $\Psi$ either remain in $\Psi$ for all time,
so converge to $\bO$ due to the assumption on the eigenvalues of $A$,
or reach the attracting sliding region in finite time.
Forward orbits on the attracting sliding region
either remain in this region for all time,
so converge to $\bO$ due to the assumption on the eigenvalues of $B$,
or reach $\tilde{\Gamma}$ at a point with $y_3 < 0$ in finite time.
By Lemma \ref{le:threeNegativeEigs}, the forward orbit
of any $y \in \tilde{\Gamma}$ with $y_3 < 0$
remains in $\Psi$ for all time, so converges to $\bO$.
Thus all forward orbits of \eqref{eq:pwl2} converge to $\bO$,
thus $\bO$ is asymptotically stable for \eqref{eq:pwl2} by Lemma \ref{le:asyStab}.
Thus $\bO$ is asymptotically stable for \eqref{eq:normalForm}, because the time scaling does not affect stability,
and the forward orbits of points with $y_1 > 0$ quickly reach $\Psi$.
Thus $\bO$ is asymptotically stable for \eqref{eq:fApprox} by Proposition \ref{pr:normalForm},
and hence $\bO$ is asymptotically stable for \eqref{eq:f} by Proposition \ref{pr:hots}.
\item
We first show that in this case all forward orbits of the linear system $\dot{y} = g^L(y)$
either reach $\tilde{\Sigma}$ at a point with $y_2 > 0$ in finite time, or converge to $\bO$ as $t \to \infty$.
The matrix $\rD g^L$ has eigenvalues $-1$ and $\frac{\alpha \pm \ri \beta}{\gamma}$,
so all forward orbits approach the two-dimensional invariant subspace $E$ associated with the complex eigenvalues.
The subspace $E$ does not coincide with $\tilde{\Sigma}$,
because the normal vector $[1,0,0]^{\sf T}$ of $\tilde{\Sigma}$ is not a left eigenvector of $\rD g^L$ for the eigenvalue $-1$,
thus orbits on $E$ repeatedly intersect $\tilde{\Sigma}$ transversally.
So if a forward orbit of $\dot{y} = g^L(y)$ does not converge to $\bO$,
it approaches $E$, so must at some time intersect $\tilde{\Sigma}$ transversally with $y_2 > 0$.

Since $\dot{y} = g^S(y)$ is a linear system with negative or complex eigenvalues,
all forward orbits of $\dot{y} = g^S(y)$ on $\tilde{\Sigma}$
either reach $\tilde{\Gamma}$ at a point with $y_3 < 0$ in finite time, or converge to $\bO$ as $t \to \infty$.

If $\Lambda$ is undefined, then the forward orbit of $(0,0,-1)$ under \eqref{eq:pwl}--\eqref{eq:gLgS} does not return to $\tilde{\Gamma}$.
In this case the orbit converges to $\bO$ as $t \to \infty$, as does every forward orbit of \eqref{eq:pwl}--\eqref{eq:gLgS}.
If $\Lambda$ is defined and $\Lambda < 1$, then the forward orbit of $(0,0,-1)$ under \eqref{eq:pwl}--\eqref{eq:gLgS}
reintersects $\tilde{\Gamma}$ at $(0,0,-\Lambda)$, then at $(0,0,-\Lambda^2)$, and so on, thus converges to $\bO$ as $t \to \infty$,
as does every forward orbit of \eqref{eq:pwl}--\eqref{eq:gLgS}.
In these cases $\bO$ is asymptotically stable for \eqref{eq:pwl}--\eqref{eq:gLgS} by Lemma \ref{le:asyStab},
so $x^*$ is asymptotically stable for \eqref{eq:f} by Propositions \ref{pr:hots} and \ref{pr:normalForm}.

Finally suppose $\Lambda > 1$.
By performing coordinate changes and time scalings
similar to those described in \S\ref{sec:red},
but not the initial truncation, we can convert \eqref{eq:f} to the form
\begin{equation}
\dot{y} = \begin{cases}
g^L(y) + o(y), & \text{until $y_1 = 0$}, \\
g^S(y) + o(y), & \text{until $y_2 = 0$}.
\end{cases}
\label{eq:pws}
\end{equation}
This is identical to \eqref{eq:pwl}, but retains higher order terms
($o(y)$ denotes terms of order greater than one).
Note that the required coordinate changes need to be nonlinear
to the obtain linear switching rule in \eqref{eq:pws},
but to first order the coordinate changes are identical to those described in \S\ref{sec:red}.
Since no truncation has been applied,
the dynamics of \eqref{eq:f} in $\Omega_L \cap \Sigma$
is topologically equivalent to the dynamics of \eqref{eq:pws} on $\Psi$.
Thus it remains to show $\bO$ is unstable for \eqref{eq:pws}.

Given small $\delta > 0$, consider the spatial blow-up
\begin{equation}
\hat{y} = \frac{y}{\delta}.
\label{eq:blowUp}
\end{equation}
In these coordinates \eqref{eq:pws} becomes
\begin{equation}
\dot{\hat{y}} = \begin{cases}
g^L(\hat{y}) + o(\delta), & \text{until $\hat{y}_1 = 0$}, \\
g^S(\hat{y}) + o(\delta), & \text{until $\hat{y}_2 = 0$}.
\end{cases}
\label{eq:pwsBlownUp}
\end{equation}
Now consider a smooth extension of \eqref{eq:pwsBlownUp} that includes $\delta \le 0$,
and notice that with $\delta = 0$ \eqref{eq:pwsBlownUp} reverts to \eqref{eq:pwl}.

If the forward orbit of $(0,0,-1)$ under \eqref{eq:pwsBlownUp}
returns to $\tilde{\Gamma}$, write its first point of return as $\left( 0, 0, -\kappa(\delta) \right)$.
Notice $\kappa$ exists and is equal to $\Lambda > 1$ when $\delta = 0$.
Thus by continuity of the semi-flow of \eqref{eq:pwsBlownUp},
there exists $\delta_0 > 0$ such that $\kappa(\delta) > \frac{\Lambda + 1}{2}$ for all $\delta \in (0,\delta_0]$.
Then for \eqref{eq:pws}, the forward orbit of $(0,0,-\delta)$ for any $\delta \in (0,\delta_0]$
first returns to $\tilde{\Gamma}$ at $\left( 0, 0, -\kappa(\delta) \delta \right)$.
For any $\delta \in (0,\delta_0]$ the orbit repeatedly intersects $\tilde{\Gamma}$,
eventually obtaining a norm of at least $\delta_0$, thus $\bO$ is unstable.
\end{enumerate}
\end{proof}

%===============================================================================
\section{Numerical computations}
\label{sec:numerics}

In this section we explore how $\Lambda$ varies with $a$, $b$, $c$, and $d$.

To evaluate $\Lambda$, we use the formula \eqref{eq:Lambda}
and compute the orbit shown bold in Fig.~\ref{fig:B} with $z = -1$.
Numerically we follow the forward orbit $\phi(t)$ of $y = (0,0,-1)$ under $\dot{y} = g^L(y)$
and search for intersections with $\tilde{\Sigma}$.
At each time step $\phi(t)$ is evaluated from an explicit formula for the flow.
If the value of $\phi_1(t)$ becomes positive before a threshold on the norm of $\| \phi(t) \|$ is exceeded,
root-finding via the secant method is used to find the first $t > 0$ at which $\hat{y} = \phi(t) \in \tilde{\Sigma}$.

The forward orbit $\hat{\phi}(t)$ of $\hat{y}$ under $\dot{y} = g^S(y)$
is then followed in the same fashion,
and if $\hat{\phi}_2(t)$ becomes negative before a threshold on the norm of $\| \hat{\phi}(t) \|$ is exceeded,
root-finding is used to determine the first $t > 0$ at which $\hat{\hat{y}} = \hat{\phi}(t) \in \tilde{\Gamma}$.
If neither threshold is reached, then $\Lambda = -\hat{\hat{y}}_3$.

The values $a$, $b$, $c$, and $d$ are given by \eqref{eq:pwlParams}, so involve some constraints.
Specifically $b > \frac{a^2}{4}$, because $\beta > 0$.
Also $d > 0$, and $d > \frac{c^2}{4}$ in the case $c > 0$, because $\lambda^S_1$ and $\lambda^S_2$ are either complex or both negative.

%%%%%%%%%%%%%%%%%%%%%%%%%%%%%%%%%%%%%%%%%%%%%%%%%%%%%%%%%%%%%%%%%%%%%%%%%%%%%%%%%%%%%%%%%%%%%%%%%%%%
\begin{figure}[b!]
\begin{center}
\includegraphics[width=7.2cm]{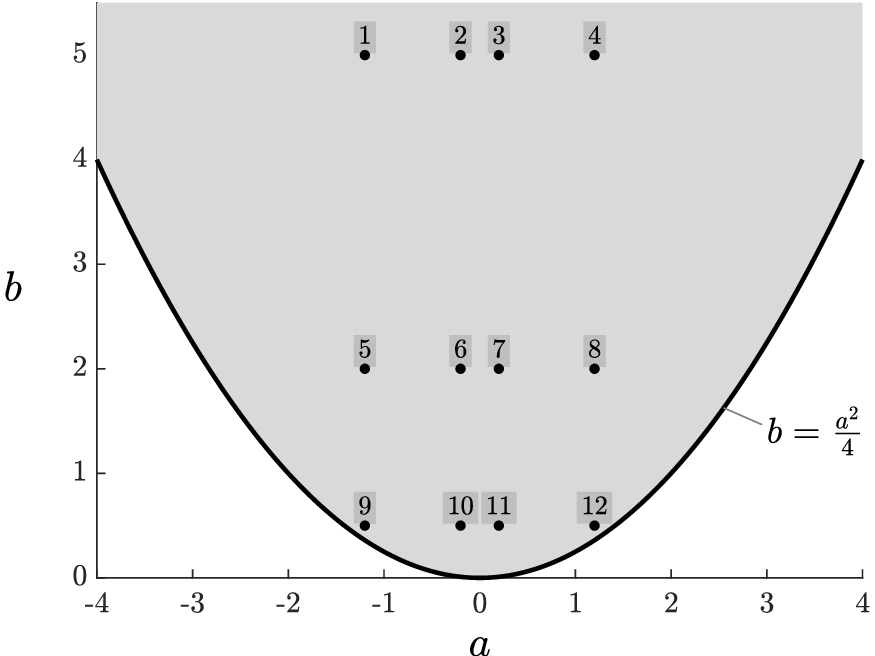}
\caption{
The twelve pairs of values $(a,b)$ used in Fig.~\ref{fig:D}.
Each pair uses $a \in \{ -1.2,-0.2,0.2,1.2 \}$ and $b \in \{ 0.5, 2, 5 \}$.
\label{fig:C}
} 
\end{center}
\end{figure}
%%%%%%%%%%%%%%%%%%%%%%%%%%%%%%%%%%%%%%%%%%%%%%%%%%%%%%%%%%%%%%%%%%%%%%%%%%%%%%%%%%%%%%%%%%%%%%%%%%%%

%%%%%%%%%%%%%%%%%%%%%%%%%%%%%%%%%%%%%%%%%%%%%%%%%%%%%%%%%%%%%%%%%%%%%%%%%%%%%%%%%%%%%%%%%%%%%%%%%%%%
\begin{figure}[b!]
\begin{center}
\includegraphics[width=15cm]{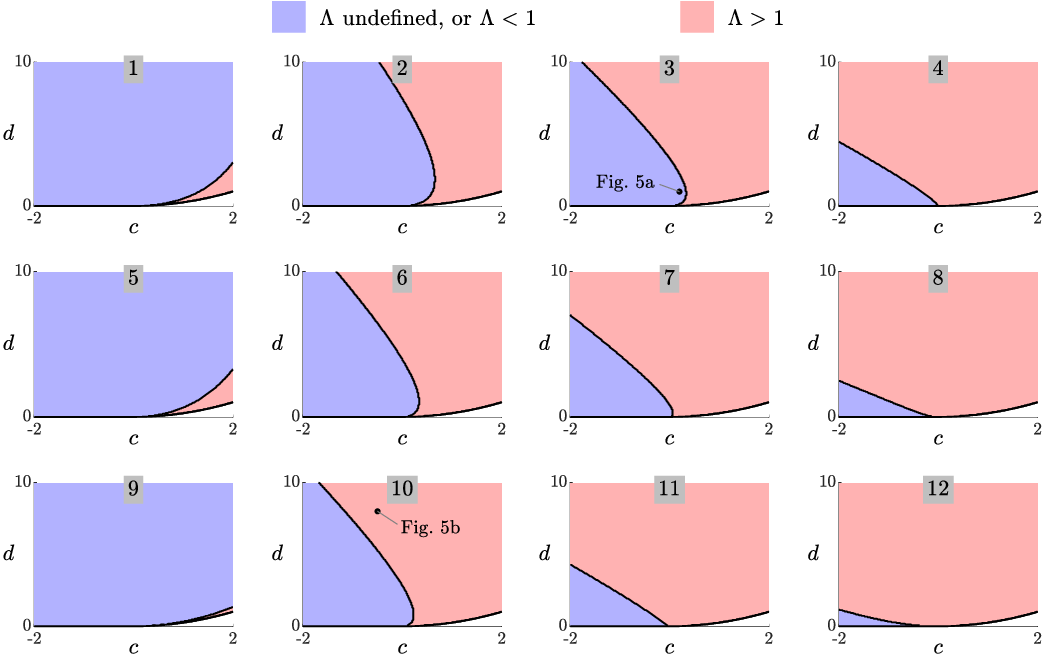}
\caption{
Regions of the $(c,d)$-plane where $\Lambda$ is undefined or $\Lambda < 1$ (blue) and $\Lambda > 1$ (red)
for the values of $a$ and $b$ indicated in Fig.~\ref{fig:C}.
Points where the algorithm described in text failed to obtain a value for $\Lambda$
are coloured blue where the norm of the orbit at some time fell below $10^{-6}$,
and red where the norm exceeded $10^6$.
In the region $c > 0$ and $d < \frac{c^2}{4}$ (white) $\Lambda$ does not apply
(refer instead to case (i) of Theorem \ref{th:main}).
\label{fig:D}
} 
\end{center}
\end{figure}
%%%%%%%%%%%%%%%%%%%%%%%%%%%%%%%%%%%%%%%%%%%%%%%%%%%%%%%%%%%%%%%%%%%%%%%%%%%%%%%%%%%%%%%%%%%%%%%%%%%%

%%%%%%%%%%%%%%%%%%%%%%%%%%%%%%%%%%%%%%%%%%%%%%%%%%%%%%%%%%%%%%%%%%%%%%%%%%%%%%%%%%%%%%%%%%%%%%%%%%%%
\begin{figure}[b!]
\begin{center}
\includegraphics[width=15cm]{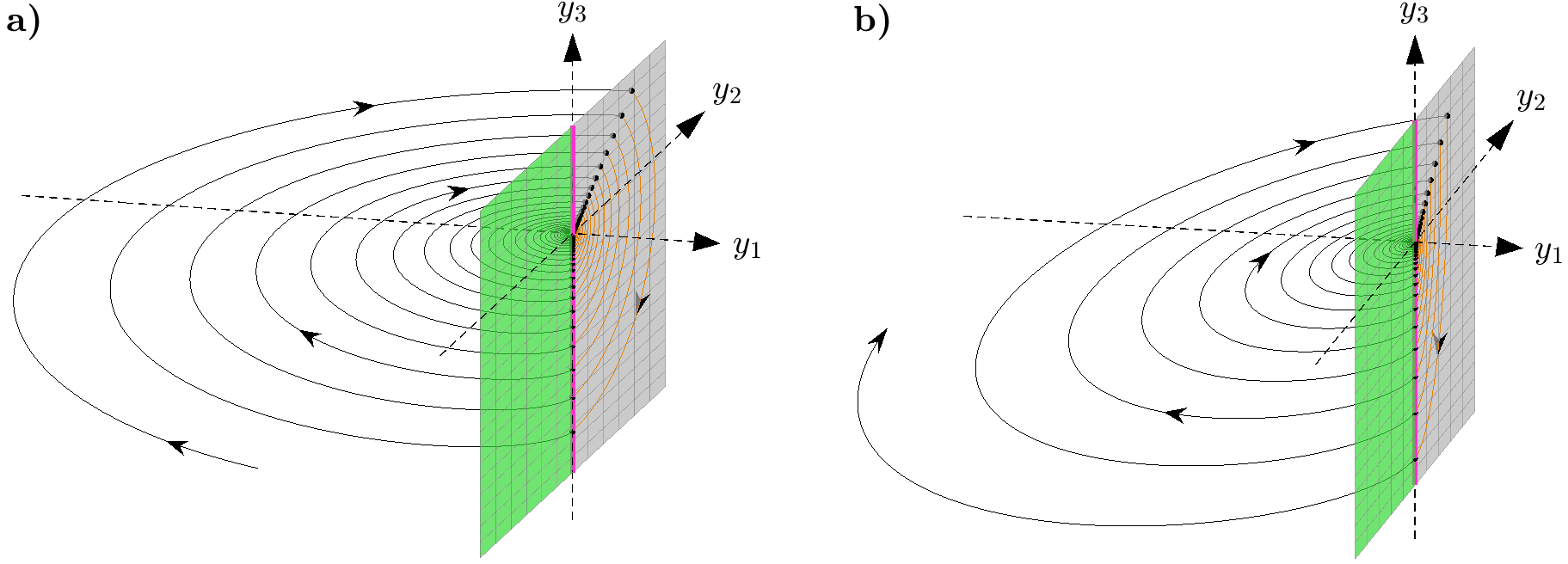}
\caption{
Sample orbits of \eqref{eq:pwl}--\eqref{eq:gLgS} with
$(a,b,c,d) = (0.2,5,0.2,1)$ in (a) and
$(a,b,c,d) = (-0.2,0.5,-0.5,8)$ in (b).
\label{fig:E}
} 
\end{center}
\end{figure}
%%%%%%%%%%%%%%%%%%%%%%%%%%%%%%%%%%%%%%%%%%%%%%%%%%%%%%%%%%%%%%%%%%%%%%%%%%%%%%%%%%%%%%%%%%%%%%%%%%%%

For twelve different combinations of $a$ and $b$ values, see Fig.~\ref{fig:C},
we show in Fig.~\ref{fig:D} the regions in the $(c,d)$-plane where $\Lambda$ is undefined or $\Lambda < 1$ (blue), and $\Lambda > 1$ (red).
These plots provide an overall impression for where $\Lambda < 1$ and where $\Lambda > 1$,  and display some interesting features.
For example, if $a > 0$ and $c > 0$,
then both the complex eigenvalues of $\rD g^L$
and the non-zero eigenvalues of $\rD g^S$ are repelling.
In this case we might expect $\Lambda > 1$,
because the forward orbit of any $y \ne \bO$ under $\dot{y} = g^L(y)$ or $\dot{y} = g^S(y)$ diverges.
Indeed the six right-most plots in Fig.~\ref{fig:D} mostly show $\Lambda > 1$ where $c > 0$,
but there are some points, such as the point indicated in plot 3, for which $\Lambda < 1$.
Here $\bO$ is an asymptotically stable equilibrium of \eqref{eq:pwl}--\eqref{eq:gLgS}, see Fig.~\ref{fig:E}a.
Each arc of the forward orbit is a segment of a repelling spiral,
yet the arcs are connected in such a way that the orbit converges to $\bO$.

Similarly if $a < 0$ and $c < 0$
then $\bO$ is asymptotically stable for each piece of \eqref{eq:pwl}--\eqref{eq:gLgS},
and indeed the six left-most plots in Fig.~\ref{fig:D} mostly show $\Lambda < 1$ where $c < 0$.
But $\bO$ can be unstable for \eqref{eq:pwl}--\eqref{eq:gLgS}, e.g.~at the point indicated in plot 10, see Fig.~\ref{fig:E}b.
Here each arc of the forward orbit is a segment of a spiral that converges to $\bO$,
yet the orbit diverges.

A closed-form expression for the boundary $\Lambda = 1$ is likely unavailable because
no closed-form expression is available for the time at which the forward orbit of $(0,0,-1)$ first returns to $\tilde{\Sigma}$.
In most of the plots in Fig.~\ref{fig:D} this boundary curves to the left as $d$ increases,
meaning that $\bO$ is less stable at larger values of $d$.

%===============================================================================
\section{Discussion}
\label{sec:conc}

In this paper we have characterised the stability of generic boundary equilibria of three-dimensional Filippov systems.
In order to apply the results to a mathematical model, one should evaluate the quantities \eqref{eq:p}--\eqref{eq:Phi} at an equilibrium $x^*$,
and compute the eigenvalues of $A$ and $B$.
Cases (i) and (ii) of Theorem \ref{th:main} address the situations
for which the stability of the $x^*$ can be deduced from these eigenvalues.
In case (iii), nearby orbits repeatedly switch between regular motion and sliding motion,
and the stability is governed by the behaviour of the forward orbit of $y = (0,0,-1)$
for the corresponding piecewise-linear system \eqref{eq:pwl}--\eqref{eq:gLgS}.
Here the regular and sliding dynamics are both rotational
and the stability of $x^*$ is governed by the sign of $\ln(\Lambda)$.

For higher-dimensional systems,
one could consider the analogous return map on the surface of visible folds.
For an $n$-dimensional system, this map will be $(n-2)$-dimensional and linear for the truncated leading-order approximation to the system.
We expect that in many cases the stability of the boundary equilibrium
will be dictated by the stability of the origin of this map.

%===============================================================================
\section*{Acknowledgements}

This work was supported by Marsden Fund contract MAU2209 managed by Royal Society Te Ap\={a}rangi.
The author thanks Jack Sandford for assistance with the numerical explorations.

\appendix

%===============================================================================
\section{Asymptotic stability for linear homogeneous semi-flows}
\label{app:asyStab}

%...............................................................................
\begin{lemma}
Let $\psi_t(y)$ be a linearly homogeneous semi-flow on a set $\Xi$ containing $\bO$.
If there exists $r > 0$ such that $\psi_t(y) \to \bO$ as $t \to \infty$ for all $y \in \Xi$ with $\| y \| < r$,
then $\bO$ is an asymptotically stable equilibrium of $\psi_t$.
\label{le:asyStab}
\end{lemma}

%...............................................................................
\begin{proof}
By \eqref{eq:linearHomogeneity}, $\psi_t(y) \to \bO$ as $t \to \infty$ for all $y \in \Xi$.
This is because for any $y \in \Xi$ there exists $\nu > 0$ such that $\| \nu y \| < r$.

Suppose for a contradiction $\bO$ is not Lyapunov stable.
Then there exists $\ee > 0$
such that for all integers $k \ge \frac{1}{\ee}$ there exists $y^{(k)} \in \Xi$ with $\left\| y^{(k)} \right\| \le \frac{1}{k}$
and $\left\| \psi_t \left( y^{(k)} \right) \right\| \ge \ee$ for some $t_k > 0$.
We can assume $\left\| \psi_{t_k} \left( y^{(k)} \right) \right\| = \ee$
and $\left\| \psi_t \left( y^{(k)} \right) \right\| \ge \frac{1}{k}$ for all $0 \le t \le t_k$.

By linear homogeneity, $\left\| \psi_t \left( k y^{(k)} \right) \right\| \ge 1$ for all $0 \le t \le t_k$,
and $t_k \to \infty$ as $k \to \infty$.
Each $k y^{(k)}$ belongs to the compact set $\mathbb{S}^2$,
thus a subsequence $k_j y^{(k_j)}$ converges to a point $z \in \mathbb{S}^2$ as $j \to \infty$.

By continuity of the semi-flow, for all $t > 0$ there exists $\eta_t > 0$ such that
$\left\| \psi_t(y) - \psi_t(z) \right\| < \frac{1}{2}$ for all $y \in \Xi$ with $\| y - z \| < \eta_t$.
Thus for all $t > 0$ there exists $j \in \mathbb{Z}$ such that $t_{k_j} > t$ and $\left\| k_j y^{(k_j)} - z \right\| < \eta_t$, hence
\begin{equation}
\left\| \psi_t(z) \right\| \ge \left\| \psi_t \left( k_j y^{(k_j)} \right) \right\|
- \left\| \psi_t \left( k_j y^{(k_j)} \right) - \psi_t(z) \right\| \ge 1 - \tfrac{1}{2} = \tfrac{1}{2}.
\nonumber
\end{equation}
Thus $\psi_t(z) \not\to \bO$ as $t \to \infty$ which is a contradiction.
\end{proof}

%===============================================================================
\section{The case of three negative eigenvalues}
\label{app:threeNegativeEigs}

%...............................................................................
\begin{lemma}
Suppose the matrix
\begin{equation}
C = \begin{bmatrix} \tau_L & 1 & 0 \\ -\sigma_L & 0 & 1 \\ \delta_L & 0 & 0 \end{bmatrix}
\label{eq:C}
\end{equation}
has eigenvalues $\lambda_1 < \lambda_2 < \lambda_3 < 0$.
Then, for the system $\dot{y} = C y$,
the forward orbit $\phi(t)$ of any point in $\tilde{\Gamma}$ with $\phi_3(0) < 0$
obeys $\phi_1(t) < 0$ for all $t > 0$.
\label{le:threeNegativeEigs}
\end{lemma}

%...............................................................................
\begin{proof}
By linearity it suffices to consider the forward orbit of $(0,0,-1)$.
Since $\tau_L = \lambda_1 + \lambda_2 + \lambda_3$,
$\sigma_L = \lambda_1 \lambda_2 + \lambda_1 \lambda_3 + \lambda_2 \lambda_3$,
and $\delta_L = \lambda_1 \lambda_2 \lambda_3$, from \eqref{eq:C} we find that
\begin{align}
v^{(1)} &= \begin{bmatrix} 1 \\ -(\lambda_2 + \lambda_3) \\ \lambda_2 \lambda_3 \end{bmatrix}, &
v^{(2)} &= \begin{bmatrix} 1 \\ -(\lambda_3 + \lambda_1) \\ \lambda_3 \lambda_1 \end{bmatrix}, &
v^{(3)} &= \begin{bmatrix} 1 \\ -(\lambda_1 + \lambda_2) \\ \lambda_1 \lambda_2 \end{bmatrix},
\nonumber
\end{align}
are eigenvectors of $C$ corresponding to $\lambda_1$, $\lambda_2$, and $\lambda_3$ respectively.
Thus
\begin{equation}
\phi(t) = k_1 \re^{\lambda_1 t} v^{(1)} + k_2 \re^{\lambda_2 t} v^{(2)} + k_3 \re^{\lambda_3 t} v^{(3)},
\nonumber
\end{equation}
for some $k_1, k_2, k_3 \in \mathbb{R}$,
and by imposing the initial condition $\phi(0) = (0,0,-1)$ we obtain
\begin{align}
k_1 &= \frac{\lambda_2 - \lambda_3}{D}, &
k_2 &= \frac{\lambda_3 - \lambda_1}{D}, &
k_3 &= \frac{\lambda_1 - \lambda_2}{D},
\nonumber
\end{align}
where $D = (\lambda_1 - \lambda_2)(\lambda_2 - \lambda_3)(\lambda_3 - \lambda_1) > 0$.
Thus $\phi_1(t) = \frac{F(t)}{D}$, where
\begin{equation}
F(t) = (\lambda_2 - \lambda_3) \re^{\lambda_1 t}
+ (\lambda_3 - \lambda_1) \re^{\lambda_2 t}
+ (\lambda_1 - \lambda_2) \re^{\lambda_3 t},
\nonumber
\end{equation}
and it remains to show $F(t) > 0$ for all $t > 0$.

Consider the function
\begin{equation}
h(u) = (\lambda_3 - \lambda_2) u^{\frac{\lambda_3 - \lambda_1}{\lambda_3 - \lambda_2}}.
\nonumber
\end{equation}
Observe $h(1) = \lambda_3 - \lambda_2$
and $\frac{d h}{d u}(1) = \lambda_3 - \lambda_1$,
thus the tangent line of $h$ at $u=1$ is
\begin{equation}
h_{\rm tang}(u) = \lambda_3 - \lambda_2 + (\lambda_3 - \lambda_1)(u-1).
\nonumber
\end{equation}
Since the exponent $\frac{\lambda_3 - \lambda_1}{\lambda_3 - \lambda_2}$ is greater than $1$,
$h$ is concave up, thus $h(u) > h_{\rm tang}(u)$ for all $u > 1$.
Thus
\begin{equation}
h \left( \re^{(\lambda_2 - \lambda_3) t} \right) > h_{\rm tang} \left( \re^{(\lambda_2 - \lambda_3) t} \right),
\nonumber
\end{equation}
for all $t > 0$, and this is algebraically equivalent to $F(t) > 0$.
\end{proof}

{\footnotesize
\bibliographystyle{unsrt}
\bibliography{Stab3dBEarXivBib}

@article{JoRa99,
	author = {Johansson, K.H. and Rantzer, A. and {\AA}str\"{o}m, K.J.},
	title = {Fast switches in relay feedback systems.},
	journal = {Automatica},
	volume = 35,
	pages = {539-552},
	year = 1999,
}

@article{TaLi12,
	author = {Tang, S. and Liang, J. and Xiao, Y. and Cheke, R.A.},
	title = {Sliding Bifurcations of {F}ilippov Two Stage Pest Control
		Models with Economic Thresholds.},
	journal = {SIAM J. Appl. Math.},
	volume = 72,
	number = 4,
	pages = {1061-1080},
	year = 2012,
}

@article{CaGi06,
	author = {Casini, P. and Giannini, O. and Vestroni, F.},
	title = {Experimental evidence of non-standard bifurcations in
		non-smooth oscillator dynamics.},
	journal = {Nonlinear Dyn.},
	volume = 46,
	number = 3,
	pages = {259-272},
	year = 2006,
}

@article{DiKo03,
	author = {di Bernardo, M. and Kowalczyk, P. and Nordmark, A.},
	title = {Sliding Bifurcations: {A} Novel Mechanism for the
		Sudden Onset of Chaos in Dry Friction Oscillators.},
	journal = {Int. J. Bifurcation Chaos},
	volume = 13,
	number = 10,
 	pages = {2935-2948},
	year = 2003,
}

@article{Sh86,
	author = {Shaw, S.W.},
	title = {On the dynamic response of a system with dry friction.},
	journal = {J. Sound. Vib.},
	volume = 108,
	number = 2,
	pages = {305-325},
	year = 1986,
}

@book{Fi88,
	author = {Filippov, A.F.},
	title = {Differential Equations with Discontinuous Righthand Sides.},
	publisher = {Kluwer Academic Publishers.},
	address = {Norwell},
	year = 1988,
}

@book{Me07,
	author = {Meiss, J.D.},
	title = {Differential Dynamical Systems.},
	publisher = {SIAM},
	address = {Philadelphia},
	year = 2007,
}

@book{Jo03,
	author = {Johansson, M.},
	title = {Piecewise Linear Control Systems.},
	series = {Lecture Notes in Control and Information Sciences.},
	volume = 284,
	publisher = {Springer-Verlag},
	address = {New York},
	year = 2003,
}

@book{Li03,
	author = {Liberzon, D.},
	title = {Switching in Systems and Control.},
	publisher = {Birkhauser},
	address = {Boston},
	year = 2003,
}

@article{LiAn09,
	author = {Lin, H. and Antsaklis, P.J.},
	title = {Stability and Stabilization of Switched Linear
		Systems: {A} Survey of Recent Results.},
	journal = {IEEE. Trans. Auto. Contr.},
	volume = 54,
	number = 2,
	pages = {308-322},
	year = 2009,
}

@inproceedings{DeRo14,
	author = {Dezuo, T. and Rodrigues, L. and Trofino, A.},
	title = {Stability Analysis of Piecewise Affine Systems with Sliding Modes.},
	booktitle = {Proceedings of the 2014 American Control Conference (ACC).},
	pages = {2005-2010},
	year = 2014,
}

@article{IeTr20,
	author = {Iervolino, R. and Trenn, S. and Vasca, F.},
	title = {Asymptotic stability of piecewise affine systems with {F}ilippov solutions
		with discontinuous piecewise {L}yapunov functions.},
	journal = {IEEE Trans. Automat. Contr.},
	volume = 66,
	number = 4,
	pages = {1513-1528},
	year = 2020,
}

@article{Su10,
	author = {Sun, Z.},
	title = {Stability of piecewise linear systems revisited.},
	journal = {Ann. Rev. Contr.},
	volume = 34,
	pages = {221-231},
	year = 2010,
}

@article{AkAr09,
	author = {Akhmet, M.U. and Aru\u{g}aslan, D.},
	title = {Bifurcation of a non-smooth planar limit cycle from a
		vertex.},
	journal = {Nonlin. Anal.},
	volume = 71,
	pages = {e2723-e2733},
	year = 2009,
}

@article{IwHa06,
	author = {Iwatani, Y. and Hara, S.},
	title = {Stability tests and stabilization for piecewise
		linear systems based on poles and zeros of subsystems.},
	journal = {Automatica},
	volume = 42,
	pages = {1685-1695},
	year = 2006,
}

@inproceedings{XuHu10,
	author = {Xu, J. and Huang, X. and Wang, S.},
	title = {Stability analysis of planar continuous piecewise linear systems.},
	booktitle = {Proceedings of the 2010 American Control Conference.},
	publisher = {IEEE},
	pages = {2505-2510},
	year = 2010,
}

@article{VaLe04,
	author = {van de Wouw, N. and Leine, R.I.},
	title = {Attractivity of Equilibrium Sets of Systems with Dry Friction.},
	journal = {Nonlin. Dyn.},
	volume = 35,
	pages = {19-39},
	year = 2004,
}

@book{DiBu08,
	author = {di Bernardo, M. and Budd, C.J. and Champneys, A.R.
		and Kowalczyk, P.},
	title = {Piecewise-smooth Dynamical Systems. Theory and Applications.},
	publisher = {Springer-Verlag},
	address = {New York},
	year = 2008,
}

@article{Si25e,
	author = {Simpson, D.J.W.},
	title = {Nonsmooth Folds as Tipping Points.},
	journal = {Chaos},
	volume = 35,
	number = 2,
	pages = {023125},
	year = 2025,
}

@article{DiNo08,
	author = {di Bernardo, M. and Nordmark, A. and Olivar, G.},
	title = {Discontinuity-induced bifurcations of equilibria in
		piecewise-smooth and impacting dynamical systems.},
	journal = {Phys. D},
	volume = 237,
	pages = {119-136},
	year = 2008,
}

@article{HoHo16,
	author = {Hogan, S.J. and Homer, M.E. and Jeffrey, M.R. and
		Szalai, R.},
	title = {Piecewise smooth dynamical systems theory:
		the case of the missing boundary equilibrium bifurcations.},
	journal = {J. Nonlin. Sci.},
	volume = 26,
	pages = {1161-1173},
	year = 2016,
}

@article{Si18d,
	author = {Simpson, D.J.W.},
	title = {A general framework for boundary equilibrium bifurcations of {F}ilippov systems.},
	journal = {Chaos},
	volume = 28,
	number = 10,
	pages = {103114},
	year = 2018,
}

@article{Si21,
	author = {Simpson, D.J.W.},
	title = {On the stability of boundary equilibria in {F}ilippov systems.},
	journal = {Commun. Pure Appl. Anal.},
	volume = 20,
	number = 9,
	pages = {3093-3111},
	year = 2021,
}

@article{ElOn15,
	author = {Eldem, V. and \"{O}ner, I.},
	title = {A note on the stability of bimodal systems in $\mathbb{R}^3$ with
		discontinuous vector fields.},
	journal = {Int. J. Contr.},
	volume = 88,
	number = 4,
	pages = {729-744},
	year = 2015,
}

@article{GoMe03,
	author = {Gon\c{c}alves, J.M. and Megretski, A. and Dahleh, M.A.},
	title = {Global Analysis of Piecewise Linear Systems Using Impact Maps
		and Surface {L}yapunov Functions.},
	journal = {IEEE Trans. Automat. Contr.},
	volume = 48,
	number = 12,
	pages = {2089-2106},
	year = 2003,
}

@book{So98,
	author = {Sontag, E.D.},
	title = {Mathematical Control Theory.},
	publisher = {Springer-Verlag},
	address = {New York},
	year = 1998,
}

@book{Je18b,
	author = {Jeffrey, M.R.},
	title = {Hidden Dynamics. The Mathematics of Switches, Decisions and Other
		Discontinuous Behaviour.},
	publisher = {Springer},
	address = {New York},
	year = 2018,
}

@book{Hi76,
	author = {Hirsch, M.W.},
	title = {Differential Topology.},
	publisher = {Springer-Verlag},
	address = {New York},
	year = 1976,
}

@article{LaSt71,
	author = {Lasota, A. and Strauss, A.},
	title = {Asymptotic Behavior for Differential Equations Which
		Cannot be Locally Linearized.},
	journal = {J. Diff. Eq.},
	volume = 10,
	pages = {152-172},
	year = 1971,
}

@book{KaHa95,
	author = {Katok, A. and Hasselblatt, B.},
	title = {Introduction to the Modern Theory of Dynamical Systems.},
	publisher = {Cambridge University Press},
	address = {New York},
	year = 1995,
}
}

\end{document}